\documentclass[a4paper,11pt]{article}
\usepackage{amsmath}
\usepackage{amsfonts}
\usepackage{amssymb}
\usepackage{amsthm}
\usepackage{graphicx}
\usepackage[all]{xy,xypic}

\date{} 
\begin{document} 
\centerline {Pseudo-Primary, Classical Prime and Pseudo-Classical Primary Elements in Lattice Modules} 
\centerline{} 

\centerline{\bf {A.~V.~Bingi$^{1}$ and C.~S.~Manjarekar$^{2}$}}
\centerline{} 

\centerline{} 

\centerline{$^{1}$ Department of Mathematics}
\centerline{St.~Xavier's College(autonomous), Mumbai-400001, India}
\centerline{$email: ashok.bingi@xaviers.edu$}

\centerline{}
\centerline{$^{2}$ Formerly Professor at Department of Mathematics}
\centerline{Shivaji University, Kolhapur-416004, India}
\centerline{$email: csmanjrekar@yahoo.co.in$} 

\newtheorem{thm}{Theorem}[section]
 \newtheorem{cor}[thm]{Corollary}
 \newtheorem{lem}[thm]{Lemma}
 \newtheorem{prop}[thm]{Proposition}
 \newtheorem{defn}[thm]{Definition}
\newtheorem{rem}[thm]{Remark}
 \newtheorem{example}[thm]{Example}
\begin{abstract}
  In this paper, we introduce the notion of pseudo-primary elements and pseudo-classical primary elements in an
   $L$-module $M$ and obtain their characterizations. The aim of the paper is to show $rad(N)\in M$, the radical of $N\in M$ is prime if $N$ is either pseudo-primary or pseudo-classical primary. Also, we study classical prime elements of an $L$-module $M$ to obtain many of its characterizations and its properties. 
\end{abstract}

\paragraph*{}

\section{Introduction}
\paragraph*{}  A multiplicative lattice $L$ is a complete lattice provided with commutative, associative and join distributive multiplication in which the largest element $1$ acts as a multiplicative identity. An element $e\in L$ is called meet principal if $a\wedge be=((a:e)\wedge b)e$ for all $a,b\in L$. An element $e\in L$ is called join principal if $(ae\vee b):e=(b:e)\vee a$ for all $a,b\in L$. An element $e\in L$ is called principal if $e$ is both meet principal and join principal. An element $a\in L$ is called compact if for  $X\subseteq L$, $a\leqslant \vee X$  implies the existence of a finite number of elements $a_1,a_2,\cdot\cdot\cdot,a_n$ in $X$ such that $a\leqslant a_1\vee a_2\vee\cdot\cdot\cdot\vee a_n$. The set of compact elements of $L$ will be denoted by $L_\ast$. If each element of $L$ is a join of compact elements of $L$ then $L$ is called a compactly generated lattice or simply a CG-lattice.  $L$ is said to be a principally generated lattice or simply a PG-lattice if each element of $L$ is the join of principal elements of $L$. Throughout this paper, $L$ denotes a compactly generated multiplicative lattice with $1$ compact in which every finite product of compact elements is compact.
       
For $a,b\in L $, $(a:b)= \vee \{x \in L \mid xb \leqslant a\} $. The radical of $a\in L$ is denoted by $\sqrt{a}$ and is defined as $\vee \{ x \in L_\ast \mid x^{n} \leqslant a$, for  some  $n\in Z_+\}$. An element $a\in L$ is said to be proper if $a<1$. A proper element $p\in L$ is called a prime element if $ab\leqslant p$ implies $a\leqslant p$ or $b\leqslant p$ where $a,b\in L$ and is called a primary element if $ab\leqslant p$ implies $a\leqslant p$ or $b^n\leqslant p$ for some $n\in Z_+$ where $a,b\in L_\ast$. A proper element $q\in L$ is called $p$-primary if $q$ is primary and $p=\sqrt{q}$ is prime.  The reader is referred to \cite{AAJ} for general background and terminology in multiplicative lattices.	
       	   
           Let $M$ be a complete lattice and $L$ be a multiplicative lattice. Then $M$ is called $L$-module or module over $L$ if there is a multiplication between elements of $L$ and $M$ written as $aB$ where $a \in L$  and $B \in M$ which satisfies the following properties: \\
           \textcircled{1}~~$(\underset{\alpha}{\vee} a_\alpha)A=\underset{\alpha}{\vee}(a_\alpha\ A)$  \\
           \textcircled{2}~~$a(\underset{\alpha}{\vee} A_\alpha)=\underset{\alpha}{\vee} (a\ A_\alpha)$  \\
           \textcircled{3}~~$(ab)A=a(bA)$  \\
           \textcircled{4}~~$1A=A$  \\
           \textcircled{5}~~$0A=O_M,  \ for \ all \ a, a_\alpha\ ,b \in L \ and \ A, A_\alpha \in M $ where $1$ is the supremum of $L$ and $0$ is  the infimum of $L$. We denote by $O_M$ and $I_M$ for the least element and the greatest element of $M$ respectively. Elements of $L$ will generally be denoted by $a,b,c,\cdot\cdot\cdot $ and elements of $M$ will generally be denoted by $A,B,C,\cdot\cdot\cdot$   
           
           Let $M$ be an $L$-module. For $N \in M$ and $a \in L$ , $(N:a) = \vee \{X \in M  \ \mid \  aX \leqslant N \}$. For $A,B \in M$, $(A:B) = \vee \{ x \in L  \ \mid  \ xB \leqslant A \} $. If $(O_M:I_M)=0$ then $M$ is called a faithful $L$-module. An $L$-module $M$ is called a multiplication lattice module if for every element $N \in M$ there exists an element $a \in L$ such that $N = aI_M$. An element $N\in M$ is called meet principal if $(b\wedge (B:N))N=bN\wedge B$ for all $b\in L, B\in M$.  An element $N\in M$ is called join principal if $b\vee (B:N)=((bN\vee B):N)$ for all $b\in L, B\in M$. An element $N\in M$ is said to be principal if $N$ is both meet principal and join principal. $M$ is said to be a PG-lattice $L$-module if each element of $M$ is the join of principal elements of $M$. An element $N\in M$ is called compact if $N\leqslant \underset{\alpha}{\vee} A_{\alpha}$ implies $N\leqslant A_{\alpha_{1}}\vee A_{\alpha_{2}}\vee\cdot\cdot\cdot\vee A_{\alpha_{n}}$ for some finite subset $\{\alpha_1,\alpha_2,\cdot\cdot\cdot,\alpha_n\}$. The set of compact elements of $M$ is denoted by $M_\ast$. If each element of $M$ is a join of compact elements of $M$ then $M$ is called a CG-lattice $L$-module.  
       An element $N\in M$ is said to be proper if $N < I_M$. A proper element $N\in M$ is said to be maximal if whenever there exists an element $B\in M$ such that $N \leqslant B$ then either $N=B$ or $B=I_M$. A proper element $N\in M$ is said to be prime if for all $a\in L$,  $X\in M$, $aX \leqslant N$ implies either $X\leqslant N$ or $aI_M \leqslant N$. A proper element $N\in M$ is said to be primary if for all $a\in L$, $X\in M$, $aX \leqslant N$ implies either $X\leqslant N$ or $a^n I_M \leqslant N$ for some $n\in Z_+$.  A proper element $N\in M$ is said to be a  radical element if $(N:I_M)=\sqrt{(N:I_M)}$ where $\sqrt{(N:I_M)}= \vee \{x\in L_\ast \mid x^n\leqslant (N:I_M) \ for \ some \ n\in Z_+\}$.  A proper element $N\in M$ is said to be semiprime if for all $a,\ b\in L$, $abI_M\leqslant N$ implies either $aI_M\leqslant N$ or $bI_M\leqslant N$. A proper element $N\in M$ is said to be $p$-prime if $N$ is prime and $p=(N:I_M)\in L$ is prime. A proper element $N\in M$ is said to be $p$-primary if $N$ is primary and $p=\sqrt{N:I_M}\in L$ is prime. A prime element $N\in M$ is said to be  minimal prime over $X\in M$ if $X\leqslant N$ and whenever there exists a prime element $Q\in M$ such that $X\leqslant Q\leqslant N$ then $Q=N$.  The reader is referred to \cite{A}, \cite{CT} and \cite{J} for general background and terminology in lattice modules.
      
   The concept of primary-like submodules in an $R$-module $M$ was introduced by H.~M.~Fazaeli et.~al.~in \cite{FR},~\cite{FR2} and \cite{FRS} which is a new generalization of a primary ideal on one hand and a generalization of a prime submodule on other hand. Also, M.~Behboodi et.~al.~in \cite{BK2} and H.~M.~Fazaeli et.~al.~in \cite{FRS} introduced the concept of weakly prime submodules and weakly primary-like submodules in an $R$-module $M$ respectively. We extend the concepts of primary-like submodules, weakly prime submodules and weakly primary-like submodules of an $R$-module $M$ to an $L$-module $M$ by introducing the notions of pseudo-primary, classical prime and pseudo-classical primary elements in a lattice module $M$ respectively and obtain their characterizations. Lattice module version of some results in ~\cite{FR}, ~\cite{FR2} are obtained and significant difference lies in the fact that results with a lot of work on principal elements of $M$ had to be developed to obtain them.  
                  
                  In the first section of this paper, we define and characterize pseudo-primary elements and pseudo-classical primary elements of a lattice module $M$. In the second section, we study and characterize classical prime elements of a lattice module $M$. In the third section, we define and characterize pseudo-classical primary elements of a lattice module $M$.  
                
 \section{Pseudo-primary elements of $M$}
       
       \paragraph*{} In an attempt of extending the concept of primary-like submodules of an $R$-module $M$ as introduced in \cite{FRS} to an $L$-module $M$, we introduce the notion of pseudo-primary elements in an $L$-module $M$. Note that instead of calling primary-like elements of $M$, we shall call such elements, pseudo-primary.  Before defining pseudo-primary element of a lattice module, we prove that $(rad(N):I_M)=\sqrt{N:I_M}$ for every $N\in M$ as proved in an $R$-module $M$ by Z.~A.~El-Bast et.~al. in \cite{ES}.           
      
       By Proposition 2 of \cite{CT}, if $(N:I_M)$ is a maximal element of $L$ then $N$ is a prime element of an $L$-module $M$. But
       from the Example 3.1 in \cite{A}, it is clear that if $(N:I_M)$ is a prime element of $L$ then $N$ need not be a prime element of an $L$-module $M$. However, this is true if $M$ is a multiplication lattice $L$-module as shown in the following lemma.
        
        \begin{lem}\label{L-C71}
        Let $N$ be a proper element of a multiplication lattice $L$-module $M$. Then $N$ is a prime element of $M$ if and only if $(N:I_M)$ is a prime element of $L$.
        \end{lem} 
        \begin{proof}
        Assume that $(N:I_M)$ is a prime element of $L$. Let $aX\leqslant N$ for $a\in L$, $X\in M$. As $M$ is a multiplication lattice $L$-module, there exists $x\in L$ such that $X=xI_M$. Then $ax\leqslant (N:I_M)$. It follows that either $a\leqslant (N:I_M)$ or $x\leqslant (N:I_M)$ which implies either
        $a\leqslant (N:I_M)$ or $xI_M=X\leqslant N$ and hence $N$ is a prime element of $M$. The converse part is clear from Proposition 3.6 of \cite{A}. 
        \end{proof} 
       
       According to the definition 3.1 in \cite{MB1}, the {\bf radical} of a proper element $N\in M$ is denoted as $rad(N)$ and is defined as the element $\wedge\{P\in M \mid P$ is  a  prime element  and $ N\leqslant P \}$. Clearly, if $N\in M$ is itself prime then $rad(N)$ is also prime. Also, in a multiplication lattice $L$-module $M$, if $(N:I_M)\in L$ is prime then by Lemma $\ref{L-C71}$, $\ rad(N)$ is also prime. In this section, we show under what other conditions $rad(N)$ is prime.
       
       Obviously, $N\leqslant rad(N)$ for all $N\in M$. From Theorem 3.6 of \cite{MB1}, it is clear that $rad(N)=N$ if $N$ is either a radical element or a prime element of a multiplication lattice $L$-module $M$. Further, in any $L$-module $M$, by Lemma 3.5 of \cite{MB1}, it follows that $\sqrt{N:I_M}\leqslant (rad(N):I_M)$ and equality holds if $M$ is a multiplication lattice $L$-module, as shown in the following theorem.
       
       \begin{thm}\label{T-C71}
       Let $L$ be a PG-lattice and $M$ be a faithful multiplication PG-lattice $L$-module with $I_M$ compact. Then $(rad(N):I_M)=\sqrt{N:I_M}$ for every proper element $N\in M$.
       \end{thm}
       \begin{proof}
       Since $M$ is a multiplication lattice $L$-module, $rad(N)=(rad(N):I_M)I_M$. But by Theorem 3.6 of \cite{MB1}, we have $rad(N)=(\sqrt{N:I_M})I_M$. Therefore $(rad(N):I_M)I_M=(\sqrt{N:I_M})I_M$. Since $I_M$ is compact, by Theorem 5 of $\cite{CT}$, we have $(rad(N):I_M)=\sqrt{N:I_M}$. 
       \end{proof}
       
       Obviously, $(rad(N):I_M)=(N:I_M)$ if $N$ is either a radical element of $M$ or a prime element of a multiplication lattice $L$-module $M$. 
       
       Now we define a pseudo-primary element of an $L$-module $M$.
       
       \begin{defn} 
       A proper element $N$ of an $L$-module $M$ is said to be {\bf pseudo-primary} if for all $X\in M$, $a\in L$, $aX\leqslant N$ implies either $a\leqslant (N:I_M)$ or $X\leqslant rad(N)$. 
       \end{defn}
       
       \begin{defn}
       An $L$-module $M$ is said to be pseudo-primary if the element $O_M$ of $M$ is itself pseudo-primary. 
       \end{defn}
       
       Clearly, every prime element $N\in M$ is pseudo-primary. The converse is true in a multiplication lattice $L$-module $M$ if $(N:I_M)\in L$ is prime, since in such a case $rad(N)=N$.
       
       \begin{lem}\label{L-C93}
                      For every $q_i\in L\ (i\in Z_+)$ we have $\underset{i\in Z_+}{\wedge}\sqrt{q_i}=\sqrt{\underset{i\in Z_+}{\wedge}q_i}$.
       \end{lem}
     \begin{proof}
      The proof is obvious.
     \end{proof}       
       
       The sufficient conditions for $rad(N\cap K)=rad(N)\cap rad(K)$ to hold are investigated by M.~E.~Moore et.~al.~in $\cite{MS}$ where $N$ and $K$ are submodules of an $R$-module $M$. The following lemma is a similar more general result which holds in a multiplication lattice $L$-module $M$. 
       
       \begin{lem}\label{L-C74}
       Let $L$ be a PG-lattice and $M$ be a faithful multiplication PG-lattice $L$ module with $I_M$ compact. Then $\underset{i\in Z_+}{\wedge}(rad(N_i))=rad({\underset{i\in  Z_+}{\wedge}N_i)}$ where $\{N_i\in M \mid i\in Z_+\}$.
       \end{lem}
       \begin{proof}
       By Theorem 3.6, Lemma 3.2 of \cite{MB1} and above
       Lemma \ref{L-C93}, we have\\
        $rad({\underset{i\in  Z_+}{\wedge}N_i)}=\bigg(\sqrt{({\underset{i\in  Z_+}{\wedge}N_i)}:I_M}\bigg)I_M=\bigg(\sqrt{{\underset{i\in  Z_+}{\wedge}}(N_i:I_M)}\bigg)I_M=(\underset{i\in Z_+}{\wedge}\sqrt{N_i:I_M})I_M$~\\
       $=\underset{i\in Z_+}{\wedge}((\sqrt{N_i:I_M})I_M)=\underset{i\in Z_+}{\wedge}(rad(N_i))$.
       \end{proof}
       
       Now we show that the meet and join of a family of ascending chain of pseudo-primary elements of $M$ are again pseudo-primary. 
       
       \begin{thm}
       Let $L$ be a PG-lattice and $M$ be a faithful multiplication PG-lattice $L$ module with $I_M$ compact. Let $\{N_i\mid i\in Z_+\}$ be a (ascending or descending) chain of pseudo-primary elements of $M$. Then 
       
       \textcircled{1} $\underset{i\in Z_+}{\wedge}N_i$ is a pseudo-primary element of $M$.
       
       \textcircled{2} $\underset{i\in Z_+}{\vee}N_i$ is a pseudo-primary element of $M$.
       \end{thm}
       \begin{proof}
       Let $N_1\leqslant N_2\leqslant\cdots \leqslant N_i\leqslant \cdots$ be an ascending chain of pseudo-primary elements of $M$.
       
       \textcircled{1}. Clearly, $(\underset{i\in Z_+}{\wedge}N_i)\neq I_M$. Let $rX\leqslant  (\underset{i\in Z_+}{\wedge}N_i)$ and $r\nleqslant  ((\underset{i\in Z_+}{\wedge}N_i):I_M)= \underset{i\in Z_+}{\wedge}(N_i:I_M)$ for $r\in L$, $X\in M$. Then $r\nleqslant (N_j:I_M)$ for some $j\in Z_+$ but $rX\leqslant N_j$ which implies $X\leqslant rad(N_j)$ as $N_j$ is pseudo-primary. Now let $N_i\neq N_j$. Then as $\{N_i\}$ is a chain, we have either $N_i<N_j$ or $N_j<N_i$. If $N_i<N_j$ then $(N_i:I_M)\leqslant (N_j:I_M)$ and accordingly $r\nleqslant (N_i:I_M)$. So as $N_i$ is pseudo-primary and $rX\leqslant N_i$, we have $X\leqslant rad(N_i)$. If $N_j<N_i$ then $X\leqslant rad(N_j)\leqslant rad(N_i)$. Thus by Lemma $\ref{L-C74}$, $X\leqslant \underset{i\in Z_+}{\wedge}(rad(N_i))=rad({\underset{i\in  Z_+}{\wedge}N_i)}$ which proves that $\underset{i\in Z_+}{\wedge}N_i$ is a pseudo-primary element of $M$. 
       
       \textcircled{2}. Since $I_M$ is compact, $(\underset{i\in Z_+}{\vee}N_i)\neq I_M$. Let $rX\leqslant (\underset{i\in Z_+}{\vee}N_i)$ and $r\nleqslant  ((\underset{i\in Z_+}{\vee}N_i):I_M)$ for $r\in L$, $X\in M$. Then since $\{N_i\}$ is a chain, we have $rX\leqslant N_j$ for some $j\in Z_+$ but $r\nleqslant (N_j:I_M)$. So as $N_j$ is pseudo-primary, we have $X\leqslant rad(N_j)\leqslant rad({\underset{i\in  Z_+}{\vee}N_i)}$ and thus $\underset{i\in Z_+}{\vee}N_i$ is a pseudo-primary element of $M$.
       \end{proof}
       
       The following theorem gives the characterization of pseudo-primary elements of an $L$-module $M$. 
       
       \begin{thm}
       Let $M$ be a CG lattice $L$-module and $N$ be a proper element of $M$. Then the following statements are equivalent:
       \begin{enumerate}
       \item[\textcircled{1}] $N$ is a pseudo-primary element of $M$. 
       
       \item[\textcircled{2}] $(N:I_M)=(N:X)$  for every proper element $X\nleqslant rad(N)$ in $M$.
         
       \item[\textcircled{3}] $(N:a)\leqslant rad(N)$ for every proper element $a>(N:I_M)$ in $L$.
       
       \item[\textcircled{4}] for every $r\in L_\ast$, $K\in M_\ast$, if $rK\leqslant N$ then either $r\leqslant (N:I_M)$ or $K\leqslant rad(N)$.
       \end{enumerate}
       \end{thm}
       \begin{proof}
       \textcircled{1}$\Longrightarrow$\textcircled{2}. Suppose \textcircled{1} holds. Let $X\in M$ be a proper element such that $X\nleqslant rad(N)$. Obviously, $(N:I_M)\leqslant (N:X)$ and $X\nleqslant N$. Let $b\leqslant (N:X)$ for $b\in L$. Then as $bX\leqslant N$, $X\nleqslant rad(N)$ and $N$ is pseudo-primary, we have $b\leqslant (N:I_M)$. So $(N:X)\leqslant (N:I_M)$ and thus $(N:I_M)=(N:X)$.
       
       \textcircled{2}$\Longrightarrow$\textcircled{3}. Suppose \textcircled{2} holds. Let $a\in L$ be a proper element such that $a>(N:I_M)$. Then $a\nleqslant (N:I_M)$. Let $X\leqslant (N:a)$ where $X\in M$ is a proper element. It follows that $a\leqslant (N:X)$. If $X\nleqslant rad(N)$ then by \textcircled{2}, $a\leqslant (N:X)$ implies that $a\leqslant (N:I_M)$ which is a contradiction and so we must have $X\leqslant rad(N)$. Thus $(N:a)\leqslant rad(N)$.
       
       \textcircled{3}$\Longrightarrow$\textcircled{4}. Suppose \textcircled{3} holds. Let $rX\leqslant N$ and $r\nleqslant (N:I_M)$ for $X\in M_\ast$, $r\in L_\ast$. Set $a=(N:I_M)\vee r$. Then as $a>(N:I_M)$ by \textcircled{3}, we have $(N:a)\leqslant rad(N)$. Also, $aX=(N:I_M)X\vee rX\leqslant N$ which implies $X\leqslant (N:a)\leqslant rad(N)$.  
       
       \textcircled{4}$\Longrightarrow$\textcircled{1}. Suppose \textcircled{4} holds. Let $aQ\leqslant N$ and $Q\nleqslant rad(N)$ for $a\in L$, $Q\in M$. As $L$ and $M$ are compactly generated, there exist $x'\in L_\ast$ and $Y,\ Y'\in M_\ast$ such that $x'\leqslant a$, $Y\leqslant Q$, $Y'\leqslant Q$ and $Y'\nleqslant rad(N)$. Let $x\in L_\ast$ be such that $x\leqslant a$. Then $(x\vee x')\in L_\ast$, $(Y\vee Y')\in M_\ast$ such that $(x\vee x')(Y\vee Y')\leqslant aQ\leqslant N$ and $(Y\vee Y')\nleqslant rad(N)$. So by \textcircled{4}, we have $(x\vee x')\leqslant (N:I_M)$ which implies $a\leqslant (N:I_M)$. Therefore $N$ is a pseudo-primary element of $M$.
       \end{proof}
              
       From the following examples, it is clear that in an $L$-module $M$ which is not a multiplication lattice $L$-module, a primary element of $M$ need not be pseudo-primary and a pseudo-primary element of $M$ need not be primary.
       
       \begin{example}\label{E-C71}
       If $R=Z$ is the ring of integers and $M=Z\oplus Z$ then $M$ is a module over $R$. Suppose $L(R)$ is the set of all ideals of $R$ and $L(M)$ is the set of all submodules of $M$. Then $L(M)$ is a lattice module over $L(R)$ but $L(M)$ is not a multiplication lattice module. If $N=(4,0)Z+(0,2)Z$ then $(N:I_M)=4Z$ and $rad(N)=(2,0)Z+(0,2)Z$. It is easy to see that $N$ is a primary element but not a pseudo-primary element of $L(M)$ (see H.~M.~Fazaeli et.~al.~\cite{FRS}).
       \end{example}
       
       \begin{example}
       If $R=Z$ is the ring of integers and $M=Z(p^\infty)=\{a/p^k + Z \mid$
       $k\in Z_+ \ and  \ a\ is \ an \ integer\}$ then $M$ is a module over $R$ where $p$ is a prime integer. Suppose $L(R)$ is the set of all ideals of $R$ and $L(M)$ is the set of all submodules of $M$. Then $L(M)$ is a lattice module over $L(R)$ but $L(M)$ is not a multiplication lattice module. Every proper element of $L(M)$ is pseudo-primary as $L(M)$ has no prime element. Further if $N=<1/p^t + Z> $ is a proper element of $L(M)$ then $(N:I_M)=0$ and so $p^i\notin (N:I_M)$ for all i and $(1/p^{i+t} + Z)\nsubseteq N, \ (i>0)$ but $p^i(1/p^{i+t}+Z)\subseteq N$. Thus every proper element of $L(M)$ is not primary. Hence every proper element of $L(M)$ is pseudo-primary but not primary (see H.~M.~Fazaeli et.~al.~\cite{FRS}).
       \end{example}
       
       The following result shows when a primary element of an $L$-module $M$ is pseudo-primary.
       
       \begin{thm}
       Every radical element of an $L$-module $M$ which is primary is pseudo-primary.  
       \end{thm}
       \begin{proof}
       the proof is obvious.
       \end{proof}      
       
       \begin{lem}\label{L-C72}
              Let $L$ be a PG-lattice and $M$ be a faithful multiplication PG-lattice $L$-module with $I_M$ compact. If a proper element $q\in L$ is primary then $qI_M$ is a primary element of $M$. 
              \end{lem}
              \begin{proof}
              As $I_M$ is compact and $q\in L$ is proper, by Theorem 5 of \cite{CT}, we have $qI_M\neq I_M$. Let $aX\leqslant qI_M$ and $a\nleqslant \sqrt{qI_M:I_M}$ for $a\in L,\ X\in M$. Then $a\nleqslant \sqrt{q}$. We may suppose that $X$ is a principal element. Assume that $((qI_M):X)\neq 1$. Then there exists a maximal element $m\in L$ such that $((qI_M):X)\leqslant m$. As $M$ is a multiplication lattice $L$-module and $m\in L$ is maximal, by Theorem 4 of \cite{CT}, two cases arise: 
              
              Case\textcircled{1}. For principal element $X\in M$, there exists a principal element $r\in L$ with $r\nleqslant m$ such that $r X=O_M$. Then $r\leqslant (O_M:X)\leqslant ((qI_M):X)\leqslant m$ which is a contradiction. 
              
              Case\textcircled{2}. There exists a principal element $Y\in M$ and a principal element $b\in L$ with $b\nleqslant m$ such that $bI_M\leqslant Y$. Then $bX\leqslant Y$, $baX\leqslant bqI_M=q(bI_M)\leqslant qY$ and $(O_M:Y)bI_M\leqslant (O_M:Y)Y=O_M$, since $Y$ is meet principal. As $M$ is faithful, it follows that $b(O_M:Y)=0$. Since $Y$ is meet principal, $(bX:Y)Y=bX$. Let $s=(bX:Y)$ then $sY=bX$ and so $asY=abX\leqslant qY$. Since $Y$ is meet principal, $abX = (abX:Y)Y=cY$ where $c=(abX:Y)$. Since $cY=abX\leqslant qY$ and $Y$ is join principal, we have $c\vee (O_M:Y)=(cY:Y)\leqslant (qY:Y)=q\vee (O_M:Y)$. So $bc\leqslant bq\leqslant q$. On the other hand, since $Y$ is join principal, $c=(abX:Y)=(asY:Y)=as\vee (O_M:Y)$ and so $abs\leqslant abs\vee b(O_M:Y)=b(as\vee (O_M:Y)) =bc\leqslant q$. If $b^n\leqslant q$ for some $n\in Z_+$ then $b^n\leqslant q\leqslant ((qI_M):X)\leqslant m$ which implies $b\leqslant m$. This is a contradiction and so $b\nleqslant \sqrt{q}$. Now as $abs\leqslant q,\ a\nleqslant \sqrt{q},\ b\nleqslant \sqrt{q}$ and $q$ is primary, we have $s\leqslant q$.  Hence $bX=sY\leqslant qY\leqslant qI_M$ which implies $b\leqslant ((qI_M):X)\leqslant m$, a contradiction. 
              
              Thus the assumption that $((qI_M):X)\neq 1$ is absurd and so we must have $((qI_M):X)=1$ which implies $X\leqslant (qI_M)$. Therefore $qI_M$ is a primary element of $M$.
              \end{proof}

              \begin{lem}\label{T-C72}
              Let $L$ be a PG-lattice and $M$ be a faithful multiplication PG-lattice $L$-module with $I_M$ compact. If a proper element $q\in L$ is primary then $qI_M$ is a pseudo-primary element of $M$. 
              \end{lem}
              \begin{proof}
              As $I_M$ is compact and $q\in L$ is proper, by Theorem 5 of $\cite{CT}$, we have $qI_M\neq I_M$. Let $aX\leqslant qI_M$ and $a\nleqslant (qI_M:I_M)$ where $a\in L,\ X\in M$. Then $a\nleqslant q$. We may suppose that $X$ is a principal element. Assume that $(rad(qI_M):X)\neq 1$. Then there exists a maximal element $m\in L$ such that $(rad(qI_M):X)\leqslant m$. As $M$ is a multiplication lattice $L$-module and $m\in L$ is maximal, by Theorem 4 of $\cite{CT}$, two cases arise: 
              
              Case\textcircled{1}. For principal element $X\in M$, there exists a principal element $r\in L$ with $r\nleqslant m$ such that $r X=O_M$. Then $r\leqslant (O_M:X)\leqslant (rad(qI_M):X)\leqslant m$ which is a contradiction. 
              
              Case\textcircled{2}. There exists a principal element $Y\in M$ and a principal element $b\in L$ with $b\nleqslant m$ such that $bI_M\leqslant Y$. Then $bX\leqslant Y$, $baX\leqslant bqI_M=q(bI_M)\leqslant qY$ and $(O_M:Y)bI_M\leqslant (O_M:Y)Y=O_M$, since $Y$ is meet principal. As $M$ is faithful, it follows that $b(O_M:Y)=0$. Since $Y$ is meet principal, $(bX:Y)Y=bX$. Let $s=(bX:Y)$ then $sY=bX$ and so $asY=abX\leqslant qY$. Since $Y$ is meet principal, $abX = (abX:Y)Y=cY$ where $c=(abX:Y)$. Since $cY=abX\leqslant qY$ and $Y$ is join principal, we have $c\vee (O_M:Y)=(cY:Y)\leqslant (qY:Y)=q\vee (O_M:Y)$. So $bc\leqslant bq\leqslant q$. On the other hand, since $Y$ is join principal, $c=(abX:Y)=(asY:Y)=as\vee (O_M:Y)$ and so $abs\leqslant abs\vee b(O_M:Y)=b(as\vee (O_M:Y)) =bc\leqslant q$. If $b^n\leqslant q$ for some $n\in Z_+$ then $b^n\leqslant q\leqslant (rad(qI_M):X)\leqslant m$ which implies $b\leqslant m$. This is a contradiction and so $b\nleqslant \sqrt{q}$. Now as $abs\leqslant q,\ a\nleqslant q,\ b\nleqslant \sqrt{q}$ and $q$ is primary, we have $s\leqslant \sqrt{ q}$. Hence by Lemma $\ref{L-C64}$, $bX=sY\leqslant \sqrt{q}Y\leqslant \sqrt{q}I_M\leqslant rad(qI_M)$ which implies $b\leqslant (rad(qI_M):X)\leqslant m$, a contradiction. 
              
              Thus the assumption that $(rad(qI_M):X)\neq 1$ is absurd and so we must have $(rad(qI_M):X)=1$ which implies $X\leqslant rad(qI_M)$. Therefore $qI_M$ is a pseudo-primary element of $M$.
              \end{proof}        
              
        The following theorem gives the characterization of pseudo-primary elements in a multiplication lattice $L$-module $M$.
       
       \begin{thm}\label{T-C73}
       Let $L$ be a PG-lattice and $M$ be a faithful multiplication PG-lattice $L$-module with $I_M$ compact.  
       For a proper element $N\in M$, the following statements are equivalent:
       
       \textcircled{1} $N$ is a pseudo-primary element of $M$. 
       
       \textcircled{2} $(N:I_M)$ is a primary element of $L$.
         
       \textcircled{3} $N=qI_M$ for some primary element $q\in L$.
       
       \textcircled{4} $N$ is a primary element of $M$.
       \end{thm}
       \begin{proof}\textcircled{1}$\Longrightarrow$\textcircled{2}. Suppose \textcircled{1} holds and let $ab\leqslant (N:I_M)$ for $a,\ b\in L$. Then $a(bI_M)\leqslant N$. As $N$ is pseudo-primary, we have either $a\leqslant (N:I_M)$ or $bI_M\leqslant rad(N)$ which implies either $a\leqslant (N:I_M)$ or $b\leqslant \sqrt{N:I_M}$, by Theorem $\ref{T-C71}$ and hence $(N:I_M)\in L$ is a primary element. 
       
       \textcircled{2}$\Longrightarrow$\textcircled{3}. Suppose \textcircled{2} holds. Since $M$ is a multiplication lattice $L$-module, we have $N=(N:I_M)I_M$ and so $N=qI_M$ for some primary element $(N:I_M)=q\in L$.
       
       \textcircled{3}$\Longrightarrow$\textcircled{1}. It is clear from Lemma $\ref{T-C72}$. 
       
       \textcircled{3}$\Longrightarrow$\textcircled{4}. It is clear from Lemma $\ref{L-C72}$. 
       
       \textcircled{4}$\Longrightarrow$\textcircled{2}. Suppose \textcircled{4} holds and let $ab\leqslant (N:I_M)$ for $a,\ b\in L$. Then as $N$ is primary and $a(bI_M)\leqslant N$, we have either $bI_M\leqslant N$ or $a\leqslant \sqrt{N:I_M}$ which implies either $b\leqslant (N:I_M)$ or $a\leqslant \sqrt{N:I_M}$. Hence $(N:I_M)\in L$ is primary.
       \end{proof}
       
       From the above Theorem \ref{T-C73}, it is clear that in a multiplication lattice $L$-module $M$, the concepts of primary and pseudo-primary elements coincide. 
       
       \begin{cor}\label{C-C72}
       Let $L$ be a PG-lattice and $M$ be a faithful multiplication PG-lattice $L$-module with $I_M$ compact.  
       If $N$ is a pseudo-primary element of $M$ then $\sqrt{N:I_M}$ is a prime element of $L$.
       \end{cor}
       \begin{proof}
       As $N\in M$ is pseudo-primary, by Theorem $\ref{T-C73}$, $(N:I_M)\in L$ is primary and hence $\sqrt{N:I_M}$ is a prime element of $L$.
       \end{proof}
       
       \begin{lem}\label{L-C62}
       Let $L$ be a PG-lattice and $M$ be a faithful multiplication PG-lattice $L$-module with $I_M$ compact. If a proper element $q\in L$ is a prime element then $qI_M$ is a prime element of $M$. 
       \end{lem}
       \begin{proof}
       As $I_M$ is compact and $q\in L$ is proper, by Theorem 5 of $\cite{CT}$, we have $qI_M\neq I_M$. Let $aX\leqslant qI_M$ and $a\nleqslant (qI_M:I_M)$ for $a\in L,\ X\in M$. Then $a\nleqslant q$. We may suppose that $X$ is a principal element. Assume that $((qI_M):X)\neq 1$. Then there exists a maximal element $m\in L$ such that $((qI_M):X)\leqslant m$. As $M$ is a multiplication lattice $L$-module and $m\in L$ is maximal, by Theorem 4 of $\cite{CT}$, two cases arise: 
              
       Case\textcircled{1}. For principal element $X\in M$, there exists a principal element $r\in L$ with $r\nleqslant m$ such that $r X=O_M$. Then $r\leqslant (O_M:X)\leqslant ((qI_M):X)\leqslant m$ which is a contradiction. 
              
       Case\textcircled{2}. There exists a principal element $Y\in M$ and a principal element $b\in L$ with $b\nleqslant m$ such that $bI_M\leqslant Y$. Then $bX\leqslant Y$, $baX\leqslant bqI_M=q(bI_M)\leqslant qY$ and $(O_M:Y)bI_M\leqslant (O_M:Y)Y=O_M$, since $Y$ is meet principal. As $M$ is faithful, it follows that $b(O_M:Y)=0$. Since $Y$ is meet principal, we have $(bX:Y)Y=bX$. Let $s=(bX:Y)$ then $sY=bX$ and so $asY=abX\leqslant qY$. Since $Y$ is meet principal, we have $abX = (abX:Y)Y=cY$ where $c=(abX:Y)$. Since $cY=abX\leqslant qY$ and $Y$ is join principal, we have $c\vee (O_M:Y)=(cY:Y)\leqslant (qY:Y)=q\vee (O_M:Y)$. So $bc\leqslant bq\leqslant q$. On the other hand, since $Y$ is join principal, we have $c=(abX:Y)=(asY:Y)=as\vee (O_M:Y)$ and so $abs\leqslant abs\vee b(O_M:Y)=b(as\vee (O_M:Y)) =bc\leqslant q$. If $b\leqslant q$ then $b\leqslant q\leqslant ((qI_M):X)\leqslant m$ which contradicts $b\nleqslant m$ and so $b\nleqslant q$. Now as $abs\leqslant q,\ a\nleqslant q,\ b\nleqslant q$ and $q$ is prime, we have $s\leqslant q$.  Hence $bX=sY\leqslant qY\leqslant (qI_M)$ which implies $b\leqslant ((qI_M):X)\leqslant m$, a contradiction.
       
        Thus the assumption that $((qI_M):X)\neq 1$ is absurd and so we must have $((qI_M):X)=1$ which implies $X\leqslant (qI_M)$. Therefore $qI_M$ is a prime element of $M$.
       \end{proof}

       The following result is the main objective of this section.
              
              \begin{thm}\label{C-C71}
              Let $L$ be a PG-lattice and $M$ be a faithful multiplication PG-lattice $L$-module with $I_M$ compact.  
              If a proper element $N\in M$ is  pseudo-primary  then $rad(N)$ is a prime element of $M$.
              \end{thm}
              \begin{proof}
              Since $N\in M$ is a pseudo-primary element of $M$, by Theorem $\ref{T-C73}$, we have $N=qI_M$ for some primary element $q\in L$. But by Theorem 3.6 of \cite{MB1}, we have $rad(N)=rad(qI_M)=(\sqrt{q})I_M$ where $\sqrt{q}\in L$ is a prime element as $q\in L$ is primary. Therefore  by Lemma \ref{L-C62}, $(\sqrt{q})I_M=rad(N)$ is a prime element of $M$.
              \end{proof} 
       
       Now, in view of the Corollary $\ref{C-C72}$ we give the following definition.
       
       \begin{defn}
       By a {\bf $p$-pseudo-primary} element $N\in M$, we mean a pseudo-primary element $N$ with $p=\sqrt{N:I_M}$ which is prime.
       \end{defn}
       
       In the Theorem $\ref{C-C71}$, we proved that $rad(N)$ is prime if $N\in M$ is pseudo-primary, using Theorem $\ref{T-C73}$. In the following theorem we prove the same result independently without using Theorem $\ref{T-C73}$.
       
       \begin{thm}\label{T-C74}
       Let $L$ be a PG-lattice and $M$ be a faithful multiplication PG-lattice $L$-module with $I_M$ compact. If $N\in M$ is a $p$-pseudo-primary element then $(N\vee pI_M)$ and hence $rad(N)$ is prime.
       \end{thm}
       \begin{proof}
       As $N\in M$ is a $p$-pseudo-primary element, $p=\sqrt{N:I_M}=(rad(N):I_M)$, by Theorem $\ref{T-C71}$. Let $P_i\in M$ be a $p_i$-prime element such that $N\leqslant P_i$. Then $p\leqslant p_i$ which implies that $(N\vee pI_M)\leqslant (P_i\vee p_iI_M)\leqslant P_i$, since $p_iI_M\leqslant P_i$. Thus whenever $P_i\in M$ is a $p_i$-prime element such that $N\leqslant P_i$, we have $(N\vee pI_M)\leqslant P_i$. It follows that $(N\vee pI_M)\leqslant rad(N)$.  As $pI_M\leqslant (N\vee pI_M)$, we have $p\leqslant ((N\vee pI_M):I_M)\leqslant (rad(N):I_M)=p$. Thus $p=((N\vee pI_M):I_M)$ is prime and hence by Lemma $\ref{L-C71}$, $(N\vee pI_M)$ is a prime element containing $N$ and therefore $rad(N)\leqslant (N\vee pI_M)$. It follows that $(N\vee pI_M)=rad(N)$ and consequently, $rad(N)$ is prime.
       \end{proof}
       
       As a consequence of Theorem \ref{L-C74}, we give following corollary without proof.
       
       \begin{cor}
       Let $L$ be a PG-lattice and $M$ be a faithful multiplication PG-lattice $L$-module with $I_M$ compact. If $N\in M$ is a $p$-pseudo-primary element then $rad(N)=rad(N\vee pI_M)$.
       \end{cor}
              
       \begin{thm}
       Let $L$ be a PG-lattice and $M$ be a faithful multiplication PG-lattice $L$-module with $I_M$ compact.  
       If $p=(N:I_M)$ and $N\in M$ is a proper element then the following statements are equivalent:
       
       \textcircled{1} $N$ is a $p$-pseudo-primary element of $M$.
       
       \textcircled{2} rad(N) is a $p$-prime element of $M$.
       
       \textcircled{3} rad(N) is a $p$-primary element of $M$.
       
       \textcircled{4} rad(N) is a $p$-pseudo-primary element of $M$.
       \end{thm}
       \begin{proof}
       \textcircled{1}$\Longrightarrow$\textcircled{2}. Suppose \textcircled{1} holds. So $N$ is pseudo-primary such that $p=\sqrt{N:I_M}$ which is prime. Then by Theorem $\ref{C-C71}$ and Theorem $\ref{T-C71}$, we have $rad(N)$ is prime such that $p=(rad(N):I_M)$ which is prime.
       
       \textcircled{2}$\Longrightarrow$\textcircled{3}. It is clear, since every prime is primary.
       
       \textcircled{3}$\Longrightarrow$\textcircled{1}. Suppose \textcircled{3} holds. So $rad(N)$ is primary such that $p=\sqrt{rad(N):I_M}=\sqrt{\sqrt{N:I_M}}=\sqrt{N:I_M}$, by Theorem $\ref{T-C71}$ and $p$ is prime. Let $aX\leqslant N$ for $a\in L$, $X\in M$. Then $aX\leqslant rad(N)$. As $rad(N)$ is primary, we have either $a\leqslant \sqrt{rad(N):I_M}=p=(N:I_M)$ or $X\leqslant rad(N)$ which proves that $N$ is $p$-pseudo-primary.
       
       \textcircled{3}$\Longleftrightarrow$\textcircled{4}. It is clear by Theorem $\ref{T-C73}$.
       \end{proof}
       
       \begin{cor}
       Let $L$ be a PG-lattice and $M$ be a faithful multiplication PG-lattice $L$-module with $I_M$ compact.  
       If a proper element $N\in M$ is pseudo-primary then the following statements are equivalent:
       
       \textcircled{1} rad(N) is a $p$-prime element of $M$.
       
       \textcircled{2} rad(N) is a $p$-primary element of $M$.
       
       \textcircled{3} rad(N) is a $p$-pseudo-primary element of $M$.
       \end{cor}
       \begin{proof}
       The proof is obvious.
       \end{proof}
       
       \begin{thm}
       Let $L$ be a PG-lattice and $M$ be a faithful multiplication PG-lattice $L$-module with $I_M$ compact.  
       If a proper element $N\in M$ is a pseudo-primary element
       then for every $K\nleqslant rad(N)$, $(N:K)$ is a primary element of $L$ and $\{\sqrt{N:K}\mid K\nleqslant rad(N)\}$ is a chain of prime elements of $L$. 
       \end{thm}
       \begin{proof} Let $K\in M$ be such that $K\nleqslant rad(N)$. So $K\nleqslant N$ and thus $\sqrt{N:K}\neq 1$. Let $ab\leqslant (N:K)$ for $a,\ b\in L$. Then as $(ab)K\leqslant N$, $K\nleqslant rad(N)$ and $N$ is pseudo-primary, we have $ab\leqslant (N:I_M)$ where $(N:I_M)$ is primary, by Theorem  $\ref{T-C73}$. It follows that either $a\leqslant (N:I_M)\leqslant (N:K)$ or $b\leqslant \sqrt{N:I_M} \leqslant \sqrt{N:K}$ and hence $(N:K)$ is a primary element of $L$. As $(N:K)\in L$ is a primary element, $\sqrt{N:K}\in L$ is a prime element. Let $X,\ Y\in M$ be such that $X\nleqslant rad(N)$ and $Y\nleqslant rad(N)$. Then $X\nleqslant N$ and $Y\nleqslant N$. So $\sqrt{N:X}\neq 1$ and $\sqrt{N:Y}\neq 1$. Without loss of generality, assume that $\sqrt{N:Y}\nleqslant \sqrt{N:X}$. Let $a\leqslant \sqrt{N:X}$. Then $a^nX\leqslant N$ for some $n\in Z_+$. Since $N$ is pseudo-primary and $X\nleqslant rad(N)$, we have $a^n\leqslant (N:I_M)\leqslant (N:Y)$ which implies $a\leqslant \sqrt{N:Y}$. Thus $\sqrt{N:X}\leqslant \sqrt{N:Y}$ and hence $\{\sqrt{N:K}\mid K\nleqslant rad(N)\}$ is a chain of prime elements of $L$.
       \end{proof}
       
       We conclude this section by proving that $rad(N)$ is prime through the concept of saturation of a pseudo-primary element $N\in M$. In view of the definition of saturation of a submodule of an $R$-module $M$ in \cite{L}, we define the saturation of an element of an $L$-module $M$ and also introduce the notion of varity of an element of an $L$-module $M$.  
       
       \begin{defn} 
       For a prime element $p\in L$ and a proper element $N\in M$, the element $S_p(N)=\vee \{X\in M \mid cX\leqslant N \ for \ some \  c\nleqslant p \}$ is called the {\bf saturation of $N$ with respect to $p$}.
       \end{defn}
       
       \begin{defn}
       For a proper element $N\in M$, the {\bf varity of $N$} is the set 
       $V(N)=\{P\in M \mid \ P \ is \ prime \ and \ N\leqslant P \}$.  Similarly, for a proper element $N\in M$, the {\bf varity of $(N:I_M)\in L$} is the set $V((N:I_M))=\{p\in L \mid \ p \ is \ prime \ and \ (N:I_M)\leqslant p \}$.
       \end{defn}
       
       Clearly, $rad(N)=\wedge V(N)$ for all $N\in M$.
       
       \begin{lem}\label{L-C75}
       In an L-module M, if a proper element $N\in M$ is pseudo-primary then $S_p(N)\leqslant rad(N)$ for all $p\in V((N:I_M))$.    
       \end{lem}
       \begin{proof}
       Let $X\leqslant S_p(N)$ for $p\in V((N:I_M))$. Then $cX\leqslant N$ for some $c\nleqslant p$ and $(N:I_M)\leqslant p$. It follows that $c\nleqslant (N:I_M)$. As $N$ is pseudo-primary, we have $X\leqslant rad(N)$ and hence $S_p(N)\leqslant rad(N)$. 
       \end{proof}
       
       \begin{thm}\label{T-C7}
       In an $L$-module $M$, if $N\in M$ is a pseudo-primary element and $S_p(N)$ is a prime element of $M$ for some $p\in V((N:I_M))$ then $rad(N)$ is a prime element of $M$.
       \end{thm}
       \begin{proof}
       Let $S_p(N)\in M$ be a prime element for some $p\in V((N:I_M))$. Clearly, $1N\leqslant N$ for $1\nleqslant p$. Therefore $N\leqslant S_p(N)$. It follows that $rad(N)\leqslant S_p(N)$ and so by Lemma $\ref{L-C75}$, we have $S_p(N)=rad(N)$ and consequently $rad(N)$ is a prime element of $M$.
       \end{proof}
       
       As compared to Theorem $\ref{C-C71}$ and Theorem $\ref{T-C74}$, in the above Theorem $\ref{T-C7}$ to prove $rad(N)$ is prime, we do not need $M$ to be a multiplication lattice $L$-module. 
       
       The following theorem gives another way to show $rad(N)$ is prime. 
       
       \begin{thm}
       Let $N$ be a proper element of an $L$-module $M$. Then $rad(N)$ is a pseudo-primary element of $M$ if and only if $rad(N)$ is prime. 
       \end{thm}
       \begin{proof}
       The proof is obvious.
       \end{proof}
       
 \section{Classical prime elements of $M$}
   
   In an attempt of extending the concept of weakly prime submodules of an $R$-module $M$ as introduced in \cite{BK2} to an $L$-module $M$, the notion of classical prime elements in an $L$-module $M$ was introduced in \cite{MB1}. In this section, we study this classical prime elements of a lattice module $M$ to obtain its characterizations and its properties. According to the definition 2.27 in \cite{MB1}, a proper element $N\in M$ is said to be {\bf classical prime} if for all $K\in M$, $a,\ b\in L$, $abK\leqslant N$ implies either $aK \leqslant N$ or $bK\leqslant N$. 
   
   \begin{defn}
   An $L$-module $M$ is said to be classical prime if the element $O_M$ of $M$ is classical prime.
   \end{defn}    
   
   Clearly, every prime element of $M$ is classical prime but the following example shows that the converse need not be true. 
   
   \begin{example}
   Let $R$ be an integral domain and $M=R\oplus R$. Then $M$ is a module over $R$. Suppose $L(R)$ is the set of all ideals of $R$ and $L(M)$ is the set of all submodules of $M$. Then $L(M)$ is a lattice module over $L(R)$. If $P$ is a non-zero prime ideal of $R$ then $P\oplus 0$ and $0\oplus P$ are classical prime elements of $L(M)$ but they are not prime elements of $L(M)$ (see M.~Behboodi et.~al.~\cite{BK2}).
   \end{example}
   
    According to \cite{MB1}, a proper element $Q$ of an $L$-module $M$ is said to be 2-absorbing if for all $a, b\in L$, $N\in M$, $abN\leqslant Q$ implies either $ab\leqslant (Q:I_M)$ or $bN\leqslant Q$ or $aN\leqslant Q$. \\
    
    Clearly, every classical prime element of $M$ is 2-absorbing.\\
   
 Now we obtain an interesting characterization of a classical prime element of a multiplication lattice $L$-module $M$.
   \begin{thm}\label{T-C76}
   Let $N$ be a proper element of a multiplication lattice $L$-module $M$.  Then the following statements are equivalent:
   
   \textcircled{1} $N$ is a classical prime element of $M$. 
   
   \textcircled{2} $(N:B)=(N:I_M)$ for every proper element $B\in M$ such that $B>N$ 
    
   \textcircled{3} $(N:X)=(N:I_M)$ for every $X\nleqslant N$ in $M$.
     
   \textcircled{4} $N=(N:a)$ for every proper element $a\in L$ such that $(N:I_M)<a$.
   \end{thm}
   \begin{proof}\textcircled{1}$\Longrightarrow$\textcircled{2}. Suppose \textcircled{1} holds. Let $B\in M$ be a proper element such that $B>N$.
   Since $M$ is multiplication lattice $L$-module, we have $B=bI_M$ for some $b\in L$.  Let $a\leqslant (N:B)$. Then $abI_M\leqslant N$. As $N$ is classical prime and $bI_M=B\nleqslant N$, we have $aI_M\leqslant N$ which implies $a\leqslant (N:I_M)$. Thus $(N:B)\leqslant (N:I_M)$. Clearly, $(N:I_M)\leqslant (N:B)$ and hence $(N:B)=(N:I_M)$. 
   
   \textcircled{2}$\Longrightarrow$\textcircled{3}. Suppose \textcircled{2} holds. Let $X\in M$ be a proper element such that $X\nleqslant N$. Let $Y=N\vee X$. Then by \textcircled{2}, for $Y>N$ we have $(N:Y)=(N:I_M)$ which implies $(N:I_M)=(N:(N\vee X))=(N:X)$.
   
   \textcircled{3}$\Longrightarrow$\textcircled{4}. Suppose \textcircled{3} holds.
   Let $a\in L$ be a proper element such that $(N:I_M)<a$. Then $a\nleqslant (N:I_M)$. Let $X\leqslant (N:a)$. So $a\leqslant (N:X)$. If $X\nleqslant N$ then by \textcircled{3}, we have $(N:X)=(N:I_M)$ which implies $a\leqslant (N:I_M)$, a contradiction. Hence we must have $X\leqslant N$. It follows that $(N:a)\leqslant N$ and thus $N=(N:a)$, since $N\leqslant (N:a)$.
   
   \textcircled{4}$\Longrightarrow$\textcircled{1}. Suppose \textcircled{4} holds. Let $rsX\leqslant N$ and $rX\nleqslant N$ for $r,\ s\in L,\ X\in M $. Set $a=(N:I_M)\vee r$. Then as $a>(N:I_M)$ by \textcircled{4}, we have $N=(N:a)$. Also, $asX=(N:I_M)sX\vee rsX\leqslant N$ and so $sX\leqslant (N:a)=N$ and hence $N$ is classical prime.
   \end{proof}
   
   \begin{cor}
   If a faithful multiplication $L$-module $M$ is classical prime then $(O_M:B)=0$ for every proper element $B\neq O_M$ of $M$.
   \end{cor}
   \begin{proof}
   As $O_M<B$ and $O_M$ is classical prime, by Theorem $\ref{T-C76}$-\textcircled{2}, for $N=O_M$ we have $(O_M:B)=(O_M:I_M)=0$, since $M$ is faithful.
   \end{proof}
   
   \begin{thm}\label{T-C77}
   Let $N$ be a proper element of a multiplication lattice $L$-module $M$.  Then the following statements are equivalent:
   \begin{enumerate}
   \item[\textcircled{1}] $N$ is a classical prime element of $M$. 
   
   \item[\textcircled{2}] For any proper elements $A$, $K$ and $B$ of $M$, if $(A:K)(B:K)\leqslant (N:K)$ then either $(A:K)\leqslant (N:K)$ or $(B:K)\leqslant (N:K)$
   \end{enumerate} 
   \end{thm}
   \begin{proof}\textcircled{1}$\Longrightarrow$\textcircled{2}. Suppose \textcircled{1} holds. For any proper elements $A$, $K$ and $B$ of $M$, let $(A:K)(B:K)\leqslant (N:K)$ and $(B:K)\nleqslant (N:K)$. Then there exists $b\in L$ such that $b\leqslant (B:K)$ and $b\nleqslant (N:K)$. Let $a\leqslant (A:K)$. Then $ab\leqslant (A:K)(B:K)\leqslant (N:K)$ which implies $abK\leqslant N$. As $N$ is classical prime and $bK\nleqslant N$, we get $aK\leqslant N$ and so $a\leqslant (N:K)$. Thus $(A:K)\leqslant (N:K)$. 
   
   \textcircled{2}$\Longrightarrow$\textcircled{1}. Suppose \textcircled{2} holds. Let $abK\leqslant N$ for $a,\ b\in L,\ K\in M $. Then $(abK:K)\leqslant (N:K)$. Since $(aK:K)(bK:K)\leqslant (abK:K)$, we have $(aK:K)(bK:K)\leqslant (N:K)$. So by \textcircled{2}, we get either $(aK:K)\leqslant (N:K)$ or $(bK:K)\leqslant (N:K)$. But as $a\leqslant (aK:K)$ and $b\leqslant (bK:K)$, we get either $a\leqslant (N:K)$ or $b\leqslant (N:K)$ which implies either  $aK\leqslant N$ or $bK\leqslant N$. Thus $N$ is classical prime.
   \end{proof}
   
   The above Theorem \ref{T-C77} is another characterization of a classical prime element of a multiplication lattice $L$-module $M$. In the next theorem, we obtain some interesting properties of a classical prime element of an $L$-module $M$.  
   
   \begin{thm}\label{T-C78}
   Let $N$ be a proper element of an $L$-module $M$.
   \begin{enumerate}
   \item[\textcircled{1}] $N$ is classical prime if and only if $(N:K)\in L$ is prime for every $K\nleqslant N$ in $M$. 
   
   \item[\textcircled{2}] If $N$ is classical prime then $(N:I_M)\in L$ is prime.
    
   \item[\textcircled{3}] If $N$ is classical prime then $(N:K)=(N:rK)$ for all proper elements $r\in L$ and $K\in M$ such that $rK\nleqslant N$.
    
   \item[\textcircled{4}] If $N$ is classical prime and $M$ is a multiplication lattice $L$-module then $\{(N:K)\mid K\nleqslant N \}$ is a chain of prime elements of $L$. 
   \end{enumerate}
   \end{thm}
   \begin{proof}
   \textcircled{1} Let a proper element $N\in M$ be classical prime and let $K\nleqslant N$ in $M$. Suppose $ab\leqslant (N:K)$ for $a,\ b\in L$. Then  $abK\leqslant N$. Since $N$ is classical prime, we have either $aK\leqslant N$ or $bK\leqslant N$ which implies either $a\leqslant (N:K)$ or $b\leqslant (N:K)$ and hence $(N:K)\in L$ is prime. Conversely, let $(N:K)\in L$ be prime for every $K\nleqslant N$ in $M$. Assume that $abK\leqslant N$ and $bK\nleqslant N$ for $a,\ b\in L$. Then $K\nleqslant N$ and so $(N:K)$ is prime. Since $ab\leqslant (N:K)$, $b\nleqslant (N:K)$ and $(N:K)$ is prime, we have $a\leqslant (N:K)$. Thus $aK\leqslant N$ and hence $N\in M$ is classical prime.
   
   \textcircled{2} It is obvious.
   
   \textcircled{3} Let $K\in M$ and $r\in L$ be proper elements such that $rK\nleqslant N$. Then $K\nleqslant N$ and so by \textcircled{1}, $(N:K)$ is prime. Assume that $a\leqslant (N:rK)$. Then $ar\leqslant (N:K)$. Since $r\nleqslant (N:K)$ and $(N:K)$ is prime, we have $a\leqslant (N:K)$. Therefore $(N:rK)\leqslant (N:K)$ and hence $(N:K)=(N:rK)$, since $(N:K)\leqslant (N:rK)$. 
   
   \textcircled{4} Let a proper element $N\in M$ be classical prime. Let $X,\ Y\in M$ be such that $X\nleqslant N$ and $Y\nleqslant N$. Then by \textcircled{1}, we have $(N:X)$ and $(N:Y)$ are prime elements of $L$. Since $M$ is a multiplication lattice $L$-module, there exists $x\in L$ such that $X=xI_M$. Without loss of generality, assume that $(N:Y)\nleqslant (N:X)$. Let $a\leqslant (N:X)$. Then $axI_M\leqslant N$. Since $N$ is classical prime and $xI_M=X\nleqslant N$, we have $aI_M\leqslant N$ which implies $a\leqslant (N:I_M)\leqslant (N:Y)$. Thus $(N:X)\leqslant (N:Y)$. Hence $\{(N:K)\mid K\nleqslant N \}$ is a chain of prime elements of $L$.
   \end{proof}
   
   In the following corollary, we show that a classical prime element $N\in M$ is prime if $N$ is primary.
   \begin{cor}
   A proper element $N$ of an $L$-module $M$ is prime if and only if $N$ is both primary and classical prime.
   \end{cor}
   \begin{proof}
   Let $N\in M$ be primary and classical prime. Assume that $rX\leqslant N$ and $X\nleqslant N$ for $r\in L$, $X\in M$. Then as $N$ is classical prime, by Theorem $\ref{T-C78}$-\textcircled{2}, we have $(N:I_M)$ is prime. Also, as $N$ is primary, we have $r^n\leqslant (N:I_M)$ for some $n\in Z_+$ which implies $r\leqslant (N:I_M)$ and hence $N$ is prime. The converse part is obvious.
   \end{proof}
   
   In a multiplication lattice $L$-module $M$, the concepts of classical prime and prime elements coincide as shown below. 
   \begin{thm}
   A proper element $N$ of a multiplication lattice $L$-module $M$ is prime if and only if $N$ is classical prime.
   \end{thm}
   \begin{proof}
   If $N$ is prime then obviously, $N$ is classical prime. Conversely, let $N$ be classical prime. Assume that $aX\leqslant N$ for $a\in L$, $X\in M$. Then $(aX:I_M)\leqslant (N:I_M)$. But $(aI_M:I_M)(X:I_M)\leqslant (aX:I_M)$ and so $(aI_M:I_M)(X:I_M)\leqslant (N:I_M)$. As $N$ is classical prime, by Theorem \ref{T-C77}, we have either $(aI_M:I_M)\leqslant (N:I_M)$ or $(X:I_M)\leqslant (N:I_M)$ which implies either $(aI_M:I_M)I_M\leqslant (N:I_M)I_M$ or $(X:I_M)I_M\leqslant (N:I_M)I_M$. Since $M$ is a multiplication lattice $L$-module, we get either $aI_M\leqslant N$ or $X\leqslant N$ and thus $N$ is prime.
   \end{proof}
   
   \begin{cor}
   Every maximal element of an $L$-module $M$ is classical prime.
   \end{cor}
   \begin{proof}
   The proof is obvious.
   \end{proof}
   
   \begin{thm}
   Let $N$ be a proper element of an $L$-module $M$ such that $(N:X)\neq (N:Y)$ implies $N=(N\vee X)\wedge (N\vee Y)$ for $X,\ Y\in M$. Then $N$ is classical prime.
   \end{thm}
   \begin{proof}
   Let $abK\leqslant N$ and $aK\nleqslant N$ for $a,\ b\in L$, $K\in M$. Then $a\leqslant (N:bK)$ and $a\nleqslant (N:K)$ which implies $(N:bK)\neq (N:K)$. So by hypothesis, $N=(N\vee bK)\wedge (N\vee K)$. Clearly, $bK\leqslant (N\vee bK)\wedge (N\vee K)=N$. Thus $N$ is classical prime.
   \end{proof}
   
   In earlier characterizations, namely Theorem $\ref{T-C76}$ and Theorem $\ref{T-C77}$, $M$ is a  multiplication lattice $L$-module. The following theorem gives the characterization of classical prime elements of an $L$-module $M$ where $M$ need not be a multiplication lattice $L$-module.
   
   \begin{thm}
   Let $M$ be a CG lattice $L$-module and $N$ be a proper element of $M$. Then the following statements are equivalent:
   
    \textcircled{1} $N$ is classical prime.
    
    \textcircled{2} for every $a,\ b\in L$, either $(N:ab)=(N:a)$ or $(N:ab)=(N:b)$.
    
    \textcircled{3} for every $r,\ s\in L_\ast$, $K\in M_\ast$, if $rsK\leqslant N$ then either $rK\leqslant N$ or 
    $sK\leqslant N$. 
   \end{thm}
   \begin{proof}\textcircled{1}$\Longrightarrow$\textcircled{2}. Suppose \textcircled{1} holds and let $a,\ b\in L$. As $ab\leqslant a$ and $ab\leqslant b$, we have $(N:a)\leqslant (N:ab)$ and $(N:b)\leqslant (N:ab)$. Now let $K\leqslant (N:ab)$. Then $abK\leqslant N$. Since $N$ is classical prime, we have either $aK\leqslant N$ or $bK\leqslant N$ which implies either $K\leqslant (N:a)$ or $K\leqslant (N:b)$. Thus either $(N:ab)\leqslant (N:a)$ or $(N:ab)\leqslant (N:b)$ and hence either $(N:ab)=(N:a)$ or $(N:ab)=(N:b)$.
   
   \textcircled{2}$\Longrightarrow$\textcircled{3}. Suppose \textcircled{2} holds. Let $rsK\leqslant N$ for $r,\ s\in L_\ast$, $K\in M_\ast$. Then by \textcircled{2}, we have either $(N:rs)=(N:r)$ or $(N:rs)=(N:s)$. So $K\leqslant (N:rs)$ implies either $K\leqslant (N:r)$ or $K\leqslant (N:s)$. Thus either $rK\leqslant N$ or $sK\leqslant N$.
   
   \textcircled{3}$\Longrightarrow$\textcircled{1}. Suppose \textcircled{3} holds.
   Let $abX\leqslant N$ and $aX\nleqslant N$ for $a,\ b\in L,\ X\in M$. As $L$ and $M$ are compactly generated, there exist $r,\ s\in L_\ast$, $Y'\in M_\ast$ such that $r\leqslant a,\ s\leqslant b$, $Y'\leqslant X$ and $rY'\nleqslant N$. Let $Y\in M_\ast$ be such that $Y\leqslant X$ which implies $sY\leqslant bX$. Then $r,\ s\in L_\ast,\ (Y\vee Y')\in M_\ast$ such that $rs(Y\vee Y')\leqslant abX\leqslant N,\ r(Y\vee Y')\nleqslant N$. So by \textcircled{3}, $s(Y\vee Y')\leqslant N$ which implies $sY\leqslant N$ and hence $bX\leqslant N$. Therefore $N$ is classical prime.
   \end{proof}
   
   In the next theorem, we show that the meet and join of a family of ascending chain of classical prime elements of $M$ are again classical prime.
   
   \begin{thm}
   Let $\{N_i\mid i\in Z_+\}$ be a (ascending or descending) chain of classical prime elements of an $L$-module $M$. Then 
          
   \textcircled{1} $\underset{i\in Z_+}{\wedge}N_i$ is a classical prime element of $M$.
          
   \textcircled{2} $\underset{i\in Z_+}{\vee}N_i$ is a classical prime element of $M$ if $I_M$ is compact.
   \end{thm}
   \begin{proof}
   Let $N_1\leqslant N_2\leqslant\cdots \leqslant N_i\leqslant\cdots$ be an ascending chain of  classical prime elements of $M$.
          
   \textcircled{1}. Clearly, $(\underset{i\in Z_+}{\wedge}N_i)\neq I_M$. Let $abQ\leqslant  (\underset{i\in Z_+}{\wedge}N_i)$ and $aQ\nleqslant  (\underset{i\in Z_+}{\wedge}N_i)$ for $a,\ b\in L$, $Q\in M$. Then $aQ\nleqslant N_j$ for some $j\in Z_+$ but $abQ\leqslant N_j$ which implies  $bQ\leqslant N_j$ as $N_j$ is a classical prime element. Now let $N_i\neq N_j$. Then as $\{N_i\}$ is a chain, we have either $N_i<N_j$ or $N_j<N_i$. If $N_i<N_j$ then as $N_i$ is a classical prime element, $abQ\leqslant N_i$ and $aQ\nleqslant N_i$, we have $bQ\leqslant N_i$. If $N_j<N_i$ then $bQ\leqslant N_j< N_i$. Thus $bQ\leqslant \underset{i\in Z_+}{\wedge}N_i$ which proves that $\underset{i\in Z_+}{\wedge}N_i$ is a classical prime element of $M$. 
          
   \textcircled{2}. Since $I_M$ is compact, $(\underset{i\in Z_+}{\vee}N_i)\neq I_M$. Let $abQ\leqslant (\underset{i\in Z_+}{\vee}N_i)$ and $aQ\nleqslant  (\underset{i\in Z_+}{\vee}N_i)$ for $a,\ b\in L$, $Q\in M$. Then as $\{N_i\}$ is a chain, we have $abQ\leqslant N_j$ for some $j\in Z_+$ but $aQ\nleqslant N_j$ which implies $bQ\leqslant N_j \leqslant ({\underset{i\in  Z_+}{\vee}N_i)}$ as $N_j$ is a classical prime element and thus $\underset{i\in Z_+}{\vee}N_i$ is a classical prime element of $M$.
   \end{proof}
   
   Clearly, every classical prime element is a semiprime element. In view of Theorem $\ref{T-C78}$-\textcircled{1}, we have the following result.
   
   \begin{thm}
   A proper element $N$ of an $L$-module $M$ is semiprime if and only if $(N:I_M)\in L$ is prime.
   \end{thm}
   \begin{proof}
   The proof is obvious.
   \end{proof}
   
   We conclude this section by the following corollary.
   
   \begin{cor}
   Let $N$ be a proper element of a multiplication lattice $L$-module $M$. Then the following statements are equivalent:
   
   \textcircled{1}. $N$ is a prime element of $M$.
   
   \textcircled{2}. $N$ is a classical prime element of $M$.
   
   \textcircled{3}. $N$ is a semiprime element of $M$.
   
   \textcircled{4}. $(N:I_M)$ is a prime element of $L$.
   \end{cor}
   \begin{proof}
   The proof is obvious.
   \end{proof}            
 
 \section{Pseudo-classical primary elements of $M$}
 
 H.~M.~Fazaeli et.~al.~in $\cite{FRS}$ introduced the concept of a weakly primary-like submodule of an $R$-module $M$. In this section, we define a similar term in lattice modules. Instead of calling weakly primary-like element of an $L$-module $M$, we shall call such elements, pseudo-classical primary.
 
 \begin{defn}
 A proper element $N$ of an $L$-module $M$ is said to be {\bf pseudo-classical primary} if for all $K\in M$, $a,\ b\in L$, $abK\leqslant N$ implies either $aK\leqslant N$ or $bK\leqslant rad(N)$. 
 \end{defn}
 
  Clearly, every classical prime element of $M$ is pseudo-classical primary.\\
 
 The following theorem gives the characterization of pseudo-classical primary elements of an $L$-module $M$. 
 
 \begin{thm}
 Let $M$ be a CG lattice $L$-module and N be a proper element of $M$. Then the following statements are equivalent:
 \begin{enumerate}
 \item[\textcircled{1}] $N$ is a pseudo-classical primary element of $M$. 
 
 \item[\textcircled{2}] $(N:X)=(N:rX)$ for every proper elements $r\in L$ and $X\in M$ such that $X\nleqslant N$ and $rX\nleqslant rad(N)$.
   
 \item[\textcircled{3}] for every $r,\ s\in L_\ast$, $K\in M_\ast$, if $rsK\leqslant N$ then either $rK\leqslant N$ or
 $sK\leqslant rad(N)$.
 \end{enumerate}
 \end{thm}
 \begin{proof}
 \textcircled{1}$\Longrightarrow$\textcircled{2}. Suppose \textcircled{1} holds. Let $X\in M$ and $r\in L$ be proper elements such that $rX\nleqslant rad(N)$ and $X\nleqslant N$. Obviously, $(N:X)\leqslant (N:rX)$. Let $b\leqslant (N:rX)$ for $b\in L$. Then as $rbX\leqslant N$, $rX\nleqslant rad(N)$ and $N$ is pseudo-classical primary, we have $bX\leqslant N$ which implies $b\leqslant (N:X)$. So $(N:rX)\leqslant (N:X)$ and thus $(N:X)=(N:rX)$.
 
 \textcircled{2}$\Longrightarrow$\textcircled{3}. Suppose \textcircled{2} holds. Let $rsK\leqslant N$ and $sK\nleqslant rad(N)$ for $K\in M_\ast$, $r,\ s\in L_\ast$. If $K\leqslant N$ then $rK\leqslant N$ and we are done. So assume that $K\nleqslant N$. Then by \textcircled{2}, we have $(N:K)=(N:sK)$. As $rsK\leqslant N$, it follows that $r\leqslant (N:sK)=(N:K)$ which implies $rK\leqslant N$.
 
 \textcircled{3}$\Longrightarrow$\textcircled{1}. Suppose \textcircled{3} holds. Let $abX\leqslant N$ and $aX\nleqslant rad(N)$ for $a,\ b\in L$, $X\in M$. As $L$ and $M$ are compactly generated, there exist $r,\ s\in L_\ast$ and $Y'\in M_\ast$ such that $r\leqslant a$, $s\leqslant b$, $Y'\leqslant X$ and $rY'\nleqslant rad(N)$. Let $Y\in M_\ast$ be such that $Y\leqslant X$ which implies $sY\leqslant bX$. Then $r,\ s\in L_\ast$, $(Y\vee Y')\in M_\ast$ such that $rs(Y\vee Y')\leqslant abX\leqslant N$ and $r(Y\vee Y')\nleqslant rad(N)$. So by \textcircled{3}, we have $s(Y\vee Y')\leqslant N$ which implies $sY\leqslant N$ and hence $bX\leqslant N$. Therefore $N$ is a pseudo-classical primary element of $M$.
 \end{proof}
 
 \begin{lem}\label{L-C76}
 Let $L$ be a PG-lattice and $M$ be a faithful multiplication PG-lattice $L$-module with $I_M$ compact. If a proper element $q\in L$ is primary then $qI_M$ is a pseudo-classical primary element of $M$. 
 \end{lem}
 \begin{proof}
 As $I_M$ is compact and $q\in L$ is proper by Theorem 5 of $\cite{CT}$ we have $qI_M\neq I_M$. Let $adX\leqslant qI_M$ and $aX\nleqslant qI_M$ for $a,\ d\in L,\ X\in M$. Then $a\nleqslant q$. We may suppose that $dX$ is a principal element. Assume that $(rad(qI_M):(dX))\neq 1$. Then there exists a maximal element $m\in L$ such that $(rad(qI_M):(dX))\leqslant m$. As $M$ is a multiplication lattice $L$-module and $m\in L$ is maximal, by Theorem 4 of $\cite{CT}$, two cases arise: 
 
 Case\textcircled{1}. For the principal element $dX\in M$, there exists a principal element $r\in L$ with $r\nleqslant m$ such that $r (dX)=O_M$. Then $r\leqslant (O_M:(dX))\leqslant (rad(qI_M):(dX))\leqslant m$ which is a contradiction. 
 
 Case\textcircled{2}. There exists a principal element $Y\in M$ and a principal element $b\in L$ with $b\nleqslant m$ such that $bI_M\leqslant Y$. Then $b(dX)\leqslant Y$, $badX\leqslant bqI_M=q(bI_M)\leqslant qY$ and $(O_M:Y)bI_M\leqslant (O_M:Y)Y=O_M$, since $Y$ is meet principal. As $M$ is faithful, it follows that $b(O_M:Y)=0$. Since $Y$ is meet principal, $(bdX:Y)Y=bdX$. Let $s=(bdX:Y)$ then $sY=bdX$ and so $asY=abdX\leqslant qY$. Since $Y$ is meet principal, $abdX = (abdX:Y)Y=cY$ where $c=(abdX:Y)$. Since $cY=abdX\leqslant qY$ and $Y$ is join principal, we have $c\vee (O_M:Y)=(cY:Y)\leqslant (qY:Y)=q\vee (O_M:Y)$. So $bc\leqslant bq\leqslant q$. On the other hand, since $Y$ is join principal, $c=(abdX:Y)=(asY:Y)=as\vee (O_M:Y)$ and so $abs\leqslant abs\vee b(O_M:Y)=b(as\vee (O_M:Y)) =bc\leqslant q$. If $b^n\leqslant q$ for some $n\in Z_+$ then $b^n\leqslant q\leqslant (rad(qI_M):(dX))\leqslant m$ which implies $b\leqslant m$. This is a contradiction and so $b\nleqslant \sqrt{q}$. Now as $abs\leqslant q,\ a\nleqslant q,\ b\nleqslant \sqrt{q}$ and $q$ is primary, we have $s\leqslant \sqrt{ q}$.  Hence by Lemma $\ref{L-C64}$, $b(dX)=sY\leqslant \sqrt{q}Y\leqslant \sqrt{q}I_M\leqslant rad(qI_M)$ which implies $b\leqslant (rad(qI_M):(dX))\leqslant m$, a contradiction. 
 
 Thus the assumption that $(rad(qI_M):(dX))\neq 1$ is absurd and so we must have $(rad(qI_M):(dX))=1$ which implies $dX\leqslant rad(qI_M)$. Therefore $qI_M$ is a pseudo-classical primary element of $M$.
 \end{proof}
  
 The following theorem gives the characterization of  pseudo-classical primary elements of a multiplication lattice $L$-module $M$.  
 \begin{thm}\label{T-C79}
 Let $L$ be a PG-lattice and $M$ be a faithful multiplication PG-lattice $L$-module with $I_M$ compact. For a proper element $N\in M$, the following statements are equivalent:
 
 \textcircled{1} $N$ is a pseudo-classical primary element of $M$.
  
 \textcircled{2} $(N:I_M)$ is a primary element of $L$.
   
 \textcircled{3} $N=qI_M$ for some primary element $q\in L$.
 
 \textcircled{4} $N$ is a primary element of $M$.
 \end{thm}
 \begin{proof}
 \textcircled{1}$\Longrightarrow$\textcircled{2}. Suppose \textcircled{1} holds and let $ab\leqslant (N:I_M)$ for $a,\ b\in L$. Then as $N$ is a pseudo-classical primary element and $abI_M\leqslant N$, by Theorem 3.6 of \cite{MB1}, we have either $aI_M\leqslant N$ or $bI_M\leqslant rad(N)=(\sqrt{N:I_M})I_M$. Since $I_M$ is compact, by Theorem 5 of $\cite{CT}$, it follows that either $a\leqslant (N:I_M)$ or $b\leqslant \sqrt{N:I_M}$. Thus $(N:I_M)\in L$ is a primary element.
 
 \textcircled{3}$\Longrightarrow$\textcircled{1}. It is clear from Lemma $\ref{L-C76}$. 
 
 Other implications are clear from Theorem $\ref{T-C73}$.
 \end{proof}
 
 From the above Theorem $\ref{T-C79}$, it is clear that in a multiplication lattice $L$-module $M$, the concepts of pseudo-classical primary and primary elements coincide. 
 
 The following result is the main objective of this section.
 
 \begin{thm}
 Let $L$ be a PG-lattice and $M$ be a faithful multiplication PG-lattice $L$-module with $I_M$ compact.  
 If a proper element $N\in M$ is pseudo-classical primary then rad(N) is a prime element of $M$.
 \end{thm}
 \begin{proof}
 The proof runs on the same lines as that of the proof of Theorem \ref{C-C71} and hence omitted.
 \end{proof} 
 
 The following theorem gives the relation between pseudo-primary and pseudo-classical primary elements of $M$.
 
 \begin{thm}\label{T-C80}
 Every pseudo-primary element of an $L$-module $M$ is pseudo-classical primary.
 \end{thm}
 \begin{proof}
 Let $a(bK)=abK\leqslant N$ for $a,\ b\in L$, $K\in M$. Then as $N$ is pseudo-primary, we have either $a\leqslant (N:I_M)\leqslant (N:K)$ or $bK\leqslant rad(N)$ which implies either $aK\leqslant N$ or $bK\leqslant rad(N)$. Thus $N\in M$ is a pseudo-classical primary element.
 \end{proof} 
 
 However, a pseudo-classical primary element of $M$ need not be pseudo-primary as shown in the following example.
 
 \begin{example}
 Consider the Example $\ref{E-C71}$. It is easy to see that $N$ is a pseudo-classical primary element but not a pseudo-primary element of $L(M)$ (see H.~M.~Fazaeli et.~al.~\cite{FRS}).
 \end{example}
 
   The relation among these defined notions in an $L$-module $M$ as discussed in this paper is summarised in the following implication figure.

   \begin{figure}[h!]
   \centering
   $\begin{xy}
   %vertices
   (15,60)*+{Maximal}="v0"; %
   (15,30)*+{Prime}="v1";(100,30)*+{Pseudo-Primary}="v2";%
   (0,15)*+{Primary}="v3";(30,15)*+{Semiprime}="v4";%
   (15,0)*+{Classical~
   Prime }="v6";(100,0)*+{Pseudo-Classical~Primary}="v7";%
   (15,-30)*+{2-absorbing}="v8"; %
   %arrows
   {\ar@{-->} "v0"; "v1"};
   {\ar@{-->} "v1"; "v3"}; %
   {\ar@{->} "v1"; "v4"};%
   {\ar@{->} "v1"; "v2"}; %
   {\ar@{->} "v2"; "v7"}; %
   {\ar@{->} "v1";"v6"}; %
   {\ar@{->} "v6"; "v4"}; %
   {\ar@{->} "v6"; "v7"}
   {\ar@{-->} "v6"; "v8"};
   
   \end{xy}$
   \caption{Inter-relation among the elements of $M$}
   \label{fig1}
   \end{figure}
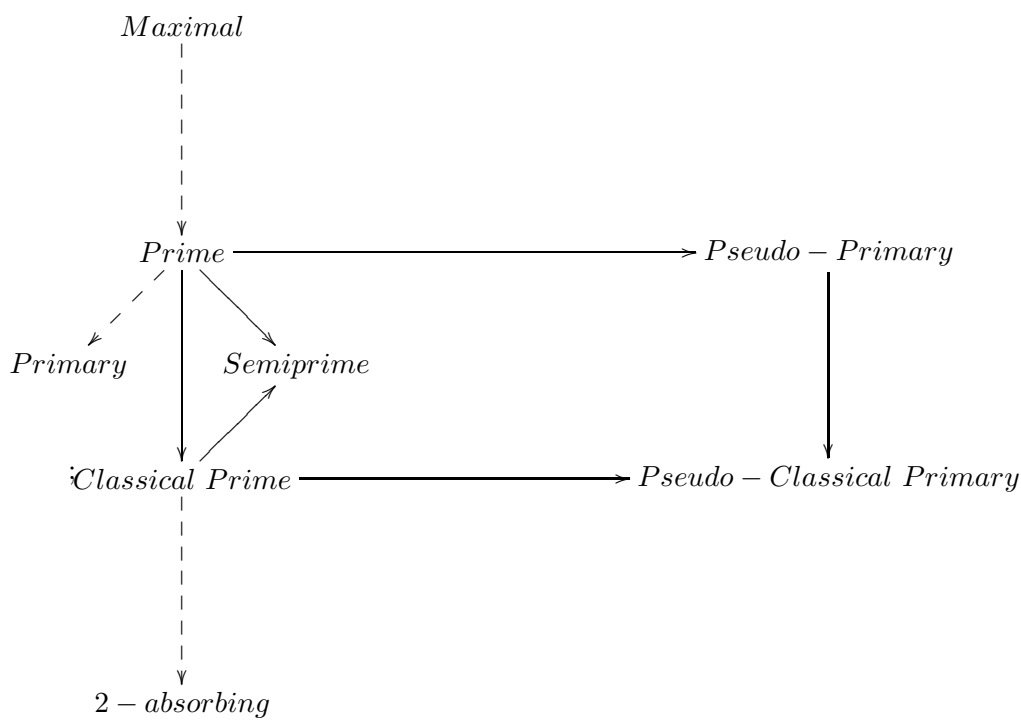  
    
 In veiw of above Theorem $\ref{T-C79}$ and Theorem $\ref{T-C80}$, we have the following corollary. 
 \begin{cor}
 Let $L$ be a PG-lattice and $M$ be a faithful multiplication PG-lattice $L$-module with $I_M$ compact.  
 For a proper element $N\in M$, the following statements are equivalent:
 
 \textcircled{1} $N$ is a pseudo-primary element of $M$. 
 
 \textcircled{2} $N$ is a pseudo-classical primary element of $M$.
  
 \textcircled{3} $(N:I_M)$ is a primary element of $L$.
   
 \textcircled{4} $N=qI_M$ for some primary element $q\in L$.
 
 \textcircled{5} $N$ is a primary element of $M$.
 \end{cor}
 \begin{proof} 
 The proof is obvious. 
 \end{proof}
 
 Thus in a multiplication lattice $L$-module $M$, the concepts of pseudo-primary, pseudo-classical primary and primary elements coincide.
 
 In the next theorem, we show that the meet and join of a family of ascending chain of pseudo-classical primary elements of $M$ are again pseudo-classical primary.
 
 \begin{thm}
 Let $L$ be a PG-lattice and $M$ be a faithful multiplication PG-lattice $L$ module with $I_M$ compact.
 Let $\{N_i\mid i\in Z_+\}$ be a chain (ascending or descending) of pseudo-classical primary elements of $M$. Then 
 
 \textcircled{1} $\underset{i\in Z_+}{\wedge}N_i$ is a pseudo-classical primary element of $M$.
 
 \textcircled{2} $\underset{i\in Z_+}{\vee}N_i$ is a pseudo-classical primary element of $M$.
 \end{thm}
 \begin{proof}
 Let $N_1\leqslant N_2\leqslant\cdots \leqslant N_i\leqslant\cdots$ be an ascending chain of pseudo-classical primary elements of $M$.
 
 \textcircled{1}. Clearly, $(\underset{i\in Z_+}{\wedge}N_i)\neq I_M$. Let $rsX\leqslant  (\underset{i\in Z_+}{\wedge}N_i)$ and $r\nleqslant  ((\underset{i\in Z_+}{\wedge}N_i):X)= \underset{i\in Z_+}{\wedge}(N_i:X)$ for $r,s\in L$ and $X\in M$. Then $r\nleqslant (N_j:X)$ for some $j\in Z_+$ but $rsX\leqslant N_j$ which implies $sX\leqslant rad(N_j)$ as $N_j$ is pseudo-classical primary. Now let $N_i\neq N_j$. Then as $\{N_i\}$ is a chain, we have either $N_i<N_j$ or $N_j<N_i$. If $N_i<N_j$ then $(N_i:X)\leqslant (N_j:X)$ and accordingly $r\nleqslant (N_i:X)$ but $rsX\leqslant N_i$ which implies $sX\leqslant rad(N_i)$ as $N_i$ is pseudo-classical primary. If $N_j<N_i$ then $sX\leqslant rad(N_j)\leqslant rad(N_i)$. Thus by Lemma $\ref{L-C74}$, $sX\leqslant \underset{i\in Z_+}{\wedge}(rad(N_i))=rad({\underset{i\in  Z_+}{\wedge}N_i)}$ which proves that $\underset{i\in Z_+}{\wedge}N_i$ is a pseudo-classical primary element of $M$. 
 
 \textcircled{2}. Since $I_M$ is compact, $(\underset{i\in Z_+}{\vee}N_i)\neq I_M$. Let $rsX\leqslant (\underset{i\in Z_+}{\vee}N_i)$ and $r\nleqslant  ((\underset{i\in Z_+}{\vee}N_i):X)$ for $r\in L$, $X\in M$. Then as $\{N_i\}$ is a chain, we have $rsX\leqslant N_j$ for some $j\in Z_+$ but $r\nleqslant (N_j:X)$ which implies $sX\leqslant rad(N_j)\leqslant rad({\underset{i\in  Z_+}{\vee}N_i)}$ as $N_j$ is pseudo-classical primary and thus $\underset{i\in Z_+}{\vee}N_i$ is a pseudo-classical primary element of $M$.
 \end{proof}
 
 We conclude this paper by finding a condition for an element $N$ of an $L$-module $M$ (which need not be a multiplication lattice $L$-module) to be pseudo-classical primary. 
 
 \begin{lem}\label{L-C77}
 In an $L$-module $M$, if a proper element $Q\in M$ is prime such that $X\leqslant Q$ then $(Q:K)\in L$ is prime such that $\sqrt{X:K}\leqslant (Q:K)$ where $K\in M$ is a proper element such that $K\nleqslant Q$.
 \end{lem}
 \begin{proof}
 As $K\nleqslant Q$, we have $K\nleqslant X$. So $(Q:K)\neq 1$ and $(X:K)\neq 1$. Let $ab\leqslant (Q:K)$ for $a,\ b\in L$. Then as $Q$ is prime and $a(bK)\leqslant Q$, we have either $a\leqslant (Q:I_M)$ or $bK\leqslant Q$ which implies either $a\leqslant (Q:I_M)\leqslant (Q:K)$ or $b\leqslant (Q:K)$. Thus $(Q:K)\in L$ is prime. Further, if $a\leqslant \sqrt{X:K}$ then $a^n\leqslant (X:K)\leqslant (Q:K)$ for some $n\in Z+$ which implies $a\leqslant (Q:K)$ and so $\sqrt{X:K}\leqslant (Q:K)$.
 \end{proof}
 
 \begin{lem}\label{L-C78}
 If $N$ is a proper element of an $L$-module $M$ then $\sqrt{N:K}\leqslant (rad(N):K)$ for every proper element $K\in M$ such that $K\nleqslant N$.
 \end{lem}
 \begin{proof}
 Let $K\in M$ be a proper element such that $K\nleqslant N$. So $(N:K)\neq 1$. Let $P\in M$ be prime such that $N\leqslant P$. Then by Lemma $\ref{L-C77}$, $(P:K)\in L$ is prime such that $\sqrt{N:K}\leqslant (P:K)$ which implies $(\sqrt{N:K})K\leqslant (P:K)K\leqslant P$. Thus whenever $P\in M$ is prime such that $N\leqslant P$, we have $(\sqrt{N:K})K\leqslant P$. It follows that $(\sqrt{N:K})K\leqslant rad(N)$. Thus $\sqrt{N:K}\leqslant (rad(N):K)$.  
 \end{proof}
 
 \begin{thm}
 Let $N$ be a proper element of an $L$-module $M$. For every proper element $K\in M$ such that $K\nleqslant N$, if $(N:K)\in L$ is a primary element then $N$ is a pseudo-classical primary element of $M$.  
 \end{thm}
 \begin{proof}
 Let $abX\leqslant N$ for $a,\ b\in L$, $X\in M$. If $X\leqslant N$ then we are done since $aX\leqslant X\leqslant N$. If $X\nleqslant N$ then by hypothesis$, (N:X)\in L$ is primary. So $ab\leqslant (N:X)$ implies that either $a\leqslant (N:X)$ or $b\leqslant \sqrt{N:X}\leqslant (rad(N):X)$, by Lemma $\ref{L-C78}$. Thus either $aX\leqslant N$ or $bX\leqslant rad(N)$ and hence $N\in M$ is a pseudo-classical primary element.
 \end{proof}

\end{document}